\documentclass[a4paper,12pt,twoside,leqno,final]{amsart}
\usepackage{amsmath}
\usepackage{amssymb}

\newtheorem{thm}{Theorem}[section]
\newtheorem{lem}[thm]{Lemma}

\newcommand{\C}{{\mathbb C}}
\newcommand{\D}{{\mathbb D}}
\newcommand{\R}{{\mathbb R}}
\newcommand{\T}{{\mathbb T}}

\newcommand{\N}{{\mathbb N}}

\newcommand{\f}{\frac}
\newcommand{\ov}{\overline}

\newcommand{\ga}{\gamma}

\newcommand{\la}{\lambda}
\newcommand{\ze}{\zeta}
\renewcommand{\th}{\theta}
\newcommand{\si}{\sigma}
\newcommand{\ph}{\varphi}

\newcommand{\Om}{\Omega}

\numberwithin{equation}{section}

\title[Quasi-squares of pseudocontinuable functions]
{Quasi-squares of pseudocontinuable functions}
\author{Konstantin M. Dyakonov}
\address{Departament de Matem\`atiques i Inform\`atica, IMUB, BGSMath, Universitat de Barcelona, Gran Via 585, E-08007 Barcelona, Spain}
\address{ICREA, Pg. Llu\'is Companys 23, E-08010 Barcelona, Spain}
\email{konstantin.dyakonov@icrea.cat}
\thanks{Supported in part by grant MTM2017-83499-P from El Ministerio de Econom\'ia y Competitividad (Spain) and grant 2017-SGR-358 from AGAUR (Generalitat de Catalunya).}

\begin{document}
\begin{abstract}
For an inner function $\th$ on the unit disk, let $K^p_\th:=H^p\cap\th\ov{H^p_0}$ be the associated star-invariant subspace of the Hardy space $H^p$. While the squaring operation $f\mapsto f^2$ maps $H^p$ into $H^{p/2}$, one cannot expect the square $f^2$ of a function $f\in K^p_\th$ to lie in $K^{p/2}_\th$. (Suffice it to note that if $f$ is a polynomial of degree $n$, then $f^2$ has degree $2n$ rather than $n$.) However, we come up with a certain \lq\lq quasi-squaring" procedure that does not have this defect. As an application, we prove an extrapolation theorem for a class of sublinear operators acting on $K^p_\th$ spaces. 
\end{abstract}

\maketitle

\section{Introduction}

Let $\T$ stand for the circle $\{\ze\in\C:|\ze|=1\}$ and $m$ for the normalized Lebesgue measure on $\T$; thus, $dm(\ze)=|d\ze|/(2\pi)$. The spaces $L^p:=L^p(\T,m)$ are then defined in the usual way and equipped with the standard norm $\|\cdot\|_p$. 
Also, for a nonnegative integer $n$, we let $\mathcal P_n$ denote the space of polynomials (in one complex variable) of degree at most $n$. 

\par Consider the following problem, stated somewhat vaguely for the time being: Given $f\in\mathcal P_n$ (with $n\in\N$), find a polynomial $g\in\mathcal P_n$ that mimics $f^2$ in the sense that $|g|$ and $|f|^2$ have the same order of magnitude on $\T$. The exact meaning of this has yet to be specified, but once that is done, we would want our \lq\lq quasi-squaring" procedure (leading from $f$ to $g$) to be fairly explicit and applicable to all $f$ in $\mathcal P_n$. 

\par Since $f^2\in\mathcal P_{2n}$, whereas $g$ is required to be in $\mathcal P_n$, there are, of course, limits to what can be expected. In particular, if $f$ has precisely $n$ zeros on $\T$, then no $g\in\mathcal P_n$ will satisfy $|g|=|f|^2$ on $\T$; nor can we hope for a two-sided estimate of the form 
\begin{equation}\label{eqn:twosidedfgont}
|f(\ze)|^2\le|g(\ze)|\le C|f(\ze)|^2,\qquad\ze\in\T,
\end{equation}
to hold with a constant $C>0$, because the right-hand inequality alone would force $g$ to be null. It turns out, however, that a slightly weaker property can be achieved. To arrive at it, we replace the problematic inequality $|g|\le C|f|^2$ in \eqref{eqn:twosidedfgont} by its $L^p$ (or rather $L^{p/2}$) version
\begin{equation}\label{eqn:rhlp}
\|g\|_{p/2}\le C\|f\|_p^2, 
\end{equation}
while leaving the other (pointwise) inequality 
\begin{equation}\label{eqn:lhineq}
|g(\ze)|\ge|f(\ze)|^2
\end{equation}
as it stands. Here, the admissible values of $p$ are those with $2<p<\infty$, as we shall see, and the constant $C=C(p)$ in \eqref{eqn:rhlp} is allowed to depend only on $p$. 
\par There is no chance \eqref{eqn:rhlp} could hold with $C=1$, as long as \eqref{eqn:lhineq} is also to be fulfilled, since otherwise it would follow that $|g|=|f|^2$ on $\T$, a condition we have already discarded as unrealistic. At the same time, our results imply the amusing fact that, for $p$ as above, the two inequalities become compatible when $C=C(p)$ is suitably large. 
\par In fact, the polynomial case hitherto discussed---and intended as a prologue---is but a toy version of the more general situation to be dealt with, the context being that of pseudocontinuable functions in Hardy spaces. Recall, to begin with, that the {\it Hardy space} $H^p$ (with $0<p<\infty$) consists of all holomorphic functions $f$ on the disk $\D:=\{z\in\C:|z|<1\}$ that satisfy
$$\sup_{0<r<1}\int_\T|f(r\ze)|^pdm(\ze)<\infty,$$
while $H^\infty$ denotes the algebra of bounded holomorphic functions on $\D$. As usual, $H^p$ functions are identified with their boundary traces on $\T$, defined in the sense of nontangential convergence almost everywhere (cf. \cite[Chapter II]{G}), and $H^p$ is viewed as a subspace of $L^p$. Recall also that a function $\th\in H^\infty$ is said to be {\it inner} if $|\th|=1$ a.e. on $\T$. We use the notation $\mathcal I$ for the set of nonconstant inner functions, and $\mathcal I_0$ for the set of inner functions $\th$ with $\th(0)=0$. 

\par Now, for $\th\in\mathcal I$, the associated {\it star-invariant} (or {\it model}) {\it subspace} $K^p_\th$ is defined by 
\begin{equation}\label{eqn:defnkpth}
K^p_\th:=H^p\cap\th\,\ov{H^p_0},\qquad1\le p\le\infty,
\end{equation}
where $H^p_0:=zH^p=\{f\in H^p:f(0)=0\}$ and the bar denotes complex conjugation. Equivalently, we have
$$K^p_\th=\{f\in H^p:\,\ov z\ov f\th\in H^p\},$$
with the understanding that the product $\ov z\ov f\th$ (and each of the three factors involved) is regarded as living a.e. on $\T$. It is well known that each $K^p_\th$ is invariant under the {\it backward shift operator}
$$\mathfrak B:f\mapsto\f{f-f(0)}z,$$ 
and conversely, every closed and nontrivial $\mathfrak B$-invariant subspace in $H^p$, with $1\le p<\infty$, is of the form $K^p_\th$ for some $\th\in\mathcal I$; see, e.g., \cite{DSS, N}. The functions belonging to some such subspace---i.e., those that are noncyclic for the backward shift---are known as {\it pseudocontinuable} functions, since they are indeed characterized by a certain \lq\lq pseudocontinuation" property. Namely, the function in question must agree a.e. on $\T$ with the boundary values of some meromorphic function of bounded characteristic in $\C\setminus(\D\cup\T)$; see \cite{DSS} for details. 

\par It is in the $K^p_\th$ setting that we actually consider our quasi-squaring problem. (The polynomial version, as discussed previously, is recovered by taking $\th(z)=z^{n+1}$, in which case $K^p_\th$ reduces to $\mathcal P_n$.) Observe, first of all, that if $\th\in\mathcal I$ and $p\ge2$, then for any function $f\in K^p_\th$, its 
square $f^2$ (and in fact the product $zf^2$) will belong to $K^{p/2}_{\th^2}$; this is essentially the best we can say of it. Passing from $p$ to $p/2$ does not really bother us---after all, this is what happens when squaring an $H^p$ function---but passing from $\th$ to $\th^2$ is what we want to avoid. Rather, we insist on keeping $\th$ intact. Accordingly, given an $f\in K^p_\th$, we seek to replace the true square, $f^2$, by a suitable \lq\lq quasi-square" $g\in K^{p/2}_\th$ for which $|g|$ is approximately of the same size as $|f|^2$ on $\T$. Precisely speaking, the properties our quasi-square should enjoy are \eqref{eqn:rhlp}, with a certain $C=C(p)$, and \eqref{eqn:lhineq}; the latter should hold for almost all $\ze\in\T$. We then prove that such a quasi-square can indeed be constructed, provided that $2<p<\infty$. 

\par In order to describe our findings more accurately, we now introduce a bit of terminology. Suppose $\mathcal E_1$ and $\mathcal E_2$ are two vector spaces consisting of functions that are defined---possibly a.e. with respect to a certain measure---on a set $\mathcal X$. A (nonlinear) operator $S:\mathcal E_1\to\mathcal E_2$ will be called {\it superquadratic} if it has the properties that 
\begin{equation}\label{eqn:defsuperone}
|S(\la f)|=|\la|^2|Sf|
\end{equation}
and 
\begin{equation}\label{eqn:defsupertwo}
|Sf|\ge|f|^2
\end{equation}
whenever $f\in\mathcal E_1$ and $\la\in\C$; the two conditions should hold everywhere---or almost everywhere---on $\mathcal X$, depending on the context. 
\par In what follows, the role of $\mathcal X$ will be played by either $\D$ or $\T$, with the \lq\lq everywhere" or \lq\lq almost everywhere" interpretation, respectively. In fact, for the spaces considered, either choice of $\mathcal X$ will do. 

\par Our main result admits a neater formulation when the underlying class of inner functions is taken to be $\mathcal I_0$, and we now state the restricted version that arises. Namely, {\it there exists a superquadratic map $S$ from $H^2$ to the weak Hardy space $H^1_w$} (a space slightly larger than $H^1$, to be defined in Section 3 below) {\it for which the following holds whenever $2<p<\infty$:}
$$S\left(K^p_\th\right)\subset K^{p/2}_\th\,\,\text{\it for each }\th\in\mathcal I_0,\qquad S\left(H^p\right)\subset H^{p/2},$$
{\it and}
$$\|Sf\|_{p/2}\le B_p\|f\|_p^2\quad\text{\it for all }f\in H^p,$$
{\it where $B_p$ is a certain} (explicit) {\it constant depending only on $p$}. In particular, if $f\in K^p_\th$ with $2<p<\infty$ and $\th\in\mathcal I_0$, then the function $g:=Sf$ is eligible as a quasi-square for $f$, since it belongs to $K^{p/2}_\th$ and has the required properties \eqref{eqn:rhlp} and \eqref{eqn:lhineq}. 

\par In addition, our construction will ensure that the image $Sf$ of every $f\in H^2\setminus\{0\}$ is an outer function. (By definition, a zero-free holomorphic function $F$ on $\D$ is {\it outer} if $\log|F|$ agrees with the harmonic extension of an $L^1$ function on $\T$.) Consequently, in our case it makes no difference whether \eqref{eqn:defsupertwo} is supposed to hold a.e. on $\T$ or everywhere on $\D$, the two conditions being equivalent. 

\par Our method relies on a preliminary result that describes the real parts of functions in $K^p_\th$. The description, which may be of independent interest, is given in Section 2 along with another auxiliary fact, to be leaned upon later. In Section 3, we state and prove our main theorem in its entirety. This includes a more complete version of the above statement involving the class $\mathcal I_0$, plus its counterpart dealing with the case of a generic $\th\in\mathcal I$. In Section 4, we discuss the endpoint values of $p$ in our quasi-squaring theorem, the emphasis being on the case $p=2$, where everything breaks down dramatically. In Section 5, we apply our quasi-squaring technique to derive an amusing extrapolation theorem for a class of sublinear operators acting on $K^p_\th$ spaces. To be more precise, we prove that if $1<p_0<\infty$ and $1\le q_0\le\infty$, and if $T$ is an operator satisfying certain hypotheses that maps $K^{p_0}_\th$ boundedly into $L^{q_0}(\mu)$ for some measure $\mu$, then $T$ is also bounded as an operator from $K^p_\th$ to $L^q(\mu)$, provided that the exponents involved are related by $p/q=p_0/q_0$ and $p_0<p<\infty$. Finally, this last theorem is discussed at some length in Section 6. In particular, we point out that our result extends a theorem of Aleksandrov from \cite{A1}, where a similar extrapolation property was established in the context of Carleson-type measures for $K^p_\th$. 

\par We conclude this introduction by looking back at the case of $\mathcal P_n$ and asking a question that puzzles us: Does our quasi-squaring construction carry over, in some form or other, to polynomials---or special classes of polynomials---in several real or complex variables?

\section{Preliminaries}

Given a function class $X$ on $\T$, we write $\text{\rm Re}\,X$ for the set of those (real-valued) functions $u$ on $\T$ that have the form $u=\text{\rm Re}\,f$ for some $f\in X$. 
\par Our current purpose is to characterize the functions $u$ in $\text{\rm Re}\,K^p_\th$ with $1\le p\le\infty$. An obvious necessary condition to be imposed is that $u\in\text{\rm Re}\,H^p$. When $1<p<\infty$, the latter simply means that $u\in L^p_\R$ (where $L^p_\R$ is the set of real-valued functions in $L^p$), the equivalence between the two conditions being due to the M. Riesz theorem; see \cite[Chapter III]{G}. For $p=1$, the assumption that $u\in\text{\rm Re}\,H^1$ can be rephrased by saying that $u$ and its nontangential maximal function 
$$u^*(\ze):=\sup\{|\mathcal Pu(z)|:\,z\in\D,\,|z-\ze|\le2(1-|z|)\},\qquad\ze\in\T$$
(where $\mathcal Pu$ is the Poisson integral of $u$), are both in $L^1_\R$; the underlying result can also be found in \cite[Chapter III]{G}. 

\begin{thm}\label{thm:realparts} Let $u\in\text{\rm Re}\,H^p$, where $1\le p\le\infty$. 
\par {\rm (a)} Suppose that $\th\in\mathcal I_0$. Then $u\in\text{\rm Re}\,K^p_\th$ if and only if $\ov zu\th\in H^p$. 
\par {\rm (b)} Suppose that $\th\in\mathcal I\setminus\mathcal I_0$. Then $u\in\text{\rm Re}\,K^p_\th$ if and only if the following two conditions hold:
\begin{equation}\label{eqn:twocond}
u\th\in H^p\quad\text{and}\quad\int_\T u\left(\frac{\th}{\th(0)}-\frac12\right)dm\in i\R.
\end{equation}
\end{thm}

\begin{proof} (a) If $u=\text{\rm Re}\,f$ for some $f\in K^p_\th$, then $u=\f12(f+\ov f)$, whence 
\begin{equation}\label{eqn:zutheta}
\ov zu\th=\f12\ov zf\th+\f12\ov z\ov f\th. 
\end{equation}
The first term on the right is in $H^p$ (because $\ov z\th=\th/z\in H^\infty$), and so is the second (because $f\in K^p_\th$). This shows that $\ov zu\th\in H^p$, proving the \lq\lq only if" part. 
\par Conversely, assume that $\ov zu\th\in H^p$. Let $f\in H^p$ be the function satisfying $\text{\rm Re}\,f=u$ (a.e. on $\T$) and $\text{\rm Im}\,f(0)=0$. Then \eqref{eqn:zutheta} is again valid, or equivalently, 
\begin{equation}\label{eqn:zuthetabis}
\ov z\ov f\th=2\ov zu\th-\ov zf\th.
\end{equation}
Our current assumption on $u$, coupled with the fact that $\ov zf\th\in H^p$, allows us to infer from \eqref{eqn:zuthetabis} that $\ov z\ov f\th\in H^p$. This means that $f\in K^p_\th$, so the \lq\lq if" part is now established as well. 
\par (b) Suppose that $u=\text{\rm Re}\,f$ for some $f\in K^p_\th$, and let $v:=\text{\rm Im}\,f$. (The functions $u$ and $v$, defined initially a.e. on $\T$, will be identified with their harmonic extensions into $\D$.) As before, we have \eqref{eqn:zutheta} and hence also \eqref{eqn:zuthetabis}; yet another way of rewriting this identity is 
$$u\th=\f12 f\th+\f12\ov f\th.$$
Here, each of the two terms on the right-hand side is in $H^p$, and therefore $u\th\in H^p$. In addition, we use the fact that $\ov z\ov f\th\in H^p$ in conjunction with \eqref{eqn:zuthetabis} to deduce that the function 
\begin{equation}\label{eqn:defng}
g:=2u\th-f\th
\end{equation}
is in $zH^p\left(=H^p_0\right)$. In particular, 
\begin{equation}\label{eqn:zeromean}
\int_{\T}g\,dm=0.
\end{equation}
Now, because 
$$\int_{\T}f\th\,dm=f(0)\th(0)=[u(0)+iv(0)]\cdot\th(0),$$
we may further rephrase \eqref{eqn:zeromean} in the form 
\begin{equation}\label{eqn:equiform}
\f2{\th(0)}\int_{\T}u\th\,dm=u(0)+iv(0).
\end{equation}
The quantity 
\begin{equation}\label{eqn:bigintegral}
\int_\T u\left(\frac{\th}{\th(0)}-\frac12\right)dm
\end{equation}
is thus equal to the (purely imaginary) number $iv(0)/2$. The necessity of \eqref{eqn:twocond} is thereby verified. 
\par Conversely, let $u\in\text{\rm Re}\,H^p$ be a function satisfying \eqref{eqn:twocond}. The value of the integral \eqref{eqn:bigintegral} being purely imaginary, say $ic$ for some $c\in\R$, we can find a function $f=u+iv\in H^p$ whose imaginary part, $v$, satisfies $v(0)=2c$. This done, we have \eqref{eqn:equiform}. Equivalently, the function $g$, defined by \eqref{eqn:defng} as before, obeys \eqref{eqn:zeromean}. We also know that $g\in H^p$, since $u\th$ and $f\th$ are both in 
$H^p$, and together with \eqref{eqn:zeromean} this means that $g$ actually belongs to $H^p_0$. Finally, we invoke the identity 
$$\ov z\ov f\th=\ov zg$$
(which coincides with \eqref{eqn:zuthetabis} and holds whenever $u=\text{\rm Re}\,f$) to conclude that $\ov z\ov f\th\in H^p$ and consequently $f\in K^p_\th$. 
\end{proof}

\medskip\noindent{\it Remarks.} (1) When proving the \lq\lq if" part in either (a) or (b), we had to produce a harmonic conjugate (say, $v$) of $u$ for which $u+iv\in K^p_\th$. In (a), the normalization $v(0)=0$ was used, but any other choice of $v$ (i.e., of the corresponding constant term) would also be fine; indeed, if $\th\in\mathcal I_0$, then $K^p_\th$ contains the constants. In (b), by contrast, the right choice is unique. 
\par (2) In \cite{DSib}, we considered the natural analogues of $K^p_\th$ spaces for the upper half-plane in place of the disk. In particular, the real parts of the functions that arise were characterized (on $\R$) by a condition similar to that in (a) above. However, the case of $\D$ turns out to be more sophisticated due to the privileged role of the point $0$, and this accounts for the dichotomy that manifests itself in Theorem \ref{thm:realparts}. 

\medskip The next fact is well known and easy to prove; see, e.g., \cite[Lemma 1]{DJAM}. When stating it, we use the standard notation $(H^\infty)^{-1}$ for the set $\{f\in H^\infty:1/f\in H^\infty\}$. 

\begin{lem}\label{lem:thetaphi} Suppose $\th$ and $\ph$ are two inner functions satisfying $\th/\ph=g/\ov g$ for some $g\in(H^\infty)^{-1}$. Then 
$$K^p_\th=gK^p_\ph\left(=\left\{gh:\,h\in K^p_\ph\right\}\right),\qquad1\le p\le\infty.$$
\end{lem}

\par In particular, this lemma applies (and will be applied) when $\ph$ is a \lq\lq Frostman shift" of $\th$, i.e., has the form 
\begin{equation}\label{eqn:frostman}
\ph=\f{\th-w}{1-\ov w\th}
\end{equation}
for some $w\in\D$. In this case, we have $g=1-\ov w\th$. 

\section{Main result}

Given a function $f\in L^1$, we write $\mathcal Hf$ for its (harmonic) conjugate, so that 
$$(\mathcal Hf)(\ze)=\text{\rm p.v.}\,\f1{2\pi}\int_{-\pi}^{\pi}f\left(\ze e^{-it}\right)\cot\f t2\,dt$$
for almost all $\ze\in\T$. Thus, dealing with a function $u\in L^1_\R$ (and using the same letter for its Poisson extension into $\D$), we may view $\mathcal Hu$ as the boundary trace of the real harmonic function $v$ on $\D$ that vanishes at $0$ and makes $u+iv$ holomorphic. 

\par It is well known (see \cite[Chapter III]{G}) that the harmonic conjugation operator $\mathcal H$ maps $L^1$ into $L^1_w$, the {\it weak} $L^1$-space, defined as the set of measurable functions $g$ on $\T$ with 
$$\sup_{\la>0}\la m\left(\{\ze\in\T:\,|g(\ze)|>\la\}\right)<\infty.$$
Another classical theorem (due to M. Riesz) asserts that $\mathcal H$ is bounded on $L^p$, and hence also on $L^p_\R$, when $1<p<\infty$. Moreover, its norm has been computed. In fact, a result of Pichorides tells us that the quantity 
$$A_p:=\sup\left\{\|\mathcal Hu\|_p:\,u\in L^p_\R,\,\|u\|_p\le1\right\}$$
equals $\tan\f{\pi}{2p}$ if $1<p\le2$ and $\cot\f{\pi}{2p}$ if $2<p<\infty$; see \cite[Theorem 3.7]{Pi}. In what follows, we also need the constants 
$$B_p:=1+A_{p/2},\qquad2<p<\infty.$$

\par Finally, we recall that the {\it Smirnov class} $N^+$ is the set of all ratios $\ph/\psi$, where $\ph$ runs through $H^\infty$ and $\psi$ through the outer functions in $H^\infty$ (see \cite[Chapter II]{G}); we then define the {\it weak Hardy space} $H^1_w$ to be $N^+\cap L^1_w$. One can find several alternative definitions---or characterizations---of $H^1_w$ in the literature, sometimes in the context of more general $H^p_w$ classes. In particular, $H^1_w$ is known to coincide with the set of holomorphic functions on $\D$ whose nontangential maximal function is in $L^1_w$; see \cite{A0, CN} for a discussion of these matters. 

\par We are now in a position to state our main result. Before doing so, we emphasize that one may interpret the term \lq\lq superquadratic," as used below, by imposing the underlying conditions \eqref{eqn:defsuperone} and \eqref{eqn:defsupertwo} either a.e. on the unit circle or inside the disk. The two interpretations are equivalent, because our maps take values in the set of outer functions. 

\begin{thm}\label{thm:quasisq} {\rm (A)} There is a superquadratic map $S:H^2\to H^1_w$ such that the image $Sf$ of every $f\in H^2\setminus\{0\}$ is an outer function, and the following holds true whenever $2<p<\infty$:
\begin{equation}\label{eqn:twoinclus}
S\left(K^p_\th\right)\subset K^{p/2}_\th\,\,\text{for each }\th\in\mathcal I_0,\qquad S\left(H^p\right)\subset H^{p/2},
\end{equation}
and
\begin{equation}\label{eqn:normsf}
\|Sf\|_{p/2}\le B_p\|f\|_p^2
\end{equation}
for all $f\in H^p$. 
\par {\rm (B)} Given $\th\in\mathcal I$, there exists a superquadratic map $S_\th:H^2\to H^1_w$ such that the image $S_\th f$ of every $f\in H^2\setminus\{0\}$ is an outer function, and the following holds true whenever $2<p<\infty$:
\begin{equation}\label{eqn:twoinclustheta}
S_\th\left(K^p_\th\right)\subset K^{p/2}_\th,\qquad S_\th\left(H^p\right)\subset H^{p/2},
\end{equation}
and 
\begin{equation}\label{eqn:normsftheta}
\|S_\th f\|_{p/2}\le B_p\left(\f{1+|\th(0)|}{1-|\th(0)|}\right)^2\|f\|_p^2
\end{equation}
for all $f\in H^p$.
\end{thm}

\par We remark that, while \eqref{eqn:normsftheta} obviously reduces to \eqref{eqn:normsf} when $\th\in\mathcal I_0$, part (A) of the theorem is not really a special case of (B). The reason is that the \lq\lq quasi-squaring" operator $S$ coming from (A) does not depend on $\th$, whereas its counterpart $S_\th$ from (B) does. At the same time, the operator $S_\th$ produced by our construction does reduce to $S$ when $\th\in\mathcal I_0$.

\medskip\noindent{\it Proof of Theorem \ref{thm:quasisq}.} (A) Given a function $f\in H^2$, we define 
$$(Sf)(z):=\int_{\T}\f{\ze+z}{\ze-z}\,|f(\ze)|^2\,dm(\ze),\qquad z\in\D.$$
In terms of the boundary values, we have
$$Sf=u+iv\quad\text{\rm a.\,e. on }\T,$$
where 
\begin{equation}\label{eqn:defumodsq}
u:=|f|^2\big|_\T
\end{equation}
and $v:=\mathcal Hu$. Because $u\in L^1$, it follows that $v\in L^1_w$ and hence $Sf\in H^1_w$. 
\par The map $S:H^2\to H^1_w$ that arises is sure to obey \eqref{eqn:defsuperone} and \eqref{eqn:defsupertwo} (these hold a.e. on $\T$, as well as everywhere on $\D$), so $S$ is superquadratic. In particular, the disk version of \eqref{eqn:defsupertwo} is verified by noting that
\begin{equation}\label{eqn:hrenovina}
|(Sf)(z)|\ge\text{\rm Re}\,(Sf)(z)=(\mathcal Pu)(z)\ge|f(z)|^2,\qquad z\in\D,
\end{equation}
where $\mathcal P$ stands for the Poisson integral operator. In addition, since a holomorphic function with positive real part is necessarily outer (see \cite[p.\,65]{G}), we infer that $Sf$ is outer whenever $f$ is non-null. Indeed, \eqref{eqn:hrenovina} tells us that $\text{\rm Re}\,(Sf)>0$ on $\D$ for any such $f$. 

\par Now suppose that $\th\in\mathcal I_0$ and $f\in K^p_\th$, where $2<p<\infty$. The corresponding function $u$, given by \eqref{eqn:defumodsq}, will then satisfy 
\begin{equation}\label{eqn:charmod}
\ov zu\th=f\cdot\ov z\ov f\th\in H^{p/2},
\end{equation}
since $f$ and $\ov z\ov f\th$ are both in $H^p$. By virtue of Theorem \ref{thm:realparts}, part (a) (see also Remark (1) following that theorem's proof), we readily deduce from \eqref{eqn:charmod} that $u\in\text{\rm Re}\,K^{p/2}_\th$ and therefore $Sf\in K^{p/2}_\th$. Thus we arrive at the first inclusion in \eqref{eqn:twoinclus}. 

\par Finally, assuming that $f$ is merely in $H^p$ (with $2<p<\infty$), we use the above-mentioned properties of the harmonic conjugation operator to obtain
\begin{equation*}
\begin{aligned}
\|Sf\|_{p/2}&\le\|u\|_{p/2}+\|v\|_{p/2}\\
&\le\left(1+A_{p/2}\right)\|u\|_{p/2}=B_p\|f\|_p^2.
\end{aligned}
\end{equation*}
This proves \eqref{eqn:normsf} and the second inclusion in \eqref{eqn:twoinclus}. 

\smallskip (B) Given $\th\in\mathcal I$, we write $w:=\th(0)$ and consider the function $g_\th:=1-\ov w\th$. Note, in particular, that $g_\th\in(H^\infty)^{-1}$. Moreover, 
\begin{equation}\label{eqn:abobel}
1-|w|\le|g_\th|\le1+|w|
\end{equation}
on $\D$. Next, we define the map $S_\th$ by putting 
$$S_\th f:=(1+|w|)\,g_\th\,S(f/g_\th),\qquad f\in H^2,$$
where $S$ is the superquadratic operator coming from part (A) above. The facts that $S_\th$ is superquadratic and maps $H^2$ into $H^1_w$ are easily deduced from the corresponding properties of $S$, coupled with \eqref{eqn:abobel}. For example, to check that $|S_\th f|\ge|f|^2$ (on $\D$) for each $f\in H^2$, one uses the estimate $|S(f/g_\th)|\ge|f/g_\th|^2$ and combines it with the right-hand inequality from \eqref{eqn:abobel}. 

\par We also have to verify that $S_\th f$ is an outer function whenever $f\in H^2\setminus\{0\}$. This is indeed true, because the functions $g_\th$ and $S(f/g_\th)$ are both outer, and so is their product. 

\par Now let $2<p<\infty$. To prove the first inclusion in \eqref{eqn:twoinclustheta}, consider the inner function $\ph$ given by \eqref{eqn:frostman} (with the current value of $w$) and note that $\ph\in\mathcal I_0$. Given $f\in K^p_\th$, we may then invoke Lemma \ref{lem:thetaphi} (and the remark following it) to infer that $f/g_\th\in K^p_\ph$. Using part (A) above with $\ph$ in place of $\th$, we further deduce that $S(f/g_\th)\in K^{p/2}_\ph$, and another application of Lemma \ref{lem:thetaphi} ensures that $g_\th S(f/g_\th)$ is in $K^{p/2}_\th$. This last function being a constant multiple of $S_\th f$, we now see that $S_\th f\in K^{p/2}_\th$, and \lq\lq half" of \eqref{eqn:twoinclustheta} is thereby established. 
 
\par Finally, to prove the remaining part of \eqref{eqn:twoinclustheta} along with the norm estimate \eqref{eqn:normsftheta}, we take an arbitrary $f\in H^p$ and proceed as follows: 
\begin{equation*}
\begin{aligned}
\|S_\th f\|_{p/2}&\le(1+|w|)\|g_\th\|_\infty\|S(f/g_\th)\|_{p/2}\\
&\le(1+|w|)^2B_p\left\|f/g_\th\right\|^2_p\le B_p\left(\f{1+|w|}{1-|w|}\right)^2\|f\|^2_p.
\end{aligned}
\end{equation*}
Here, we have combined \eqref{eqn:abobel} and \eqref{eqn:normsf}, the latter being applied with $f/g_\th$ in place of $f$. The proof is complete. \qed

\medskip We conclude with a brief remark concerning the relation \eqref{eqn:charmod}, which (in conjunction with Theorem \ref{thm:realparts}) played a key role in the above proof. Namely, the condition $\ov zu\th\in H^{p/2}$ is actually known to characterize the nonnegative functions $u$ on $\T$ that are writable as $|f|^2$ for some $f\in K^p_\th$. Various versions---and an extension---of this result can be found in \cite{DSib}, \cite[Lemma 5]{DJAM} and \cite[Theorem 1.1]{DAMP}. In the polynomial case, when $\th=z^{n+1}$, one recovers the classical Fej\'er--Riesz representation theorem for nonnegative trigonometric polynomials; see, e.g., \cite[p.\,26]{Sim}. 

\section{The endpoint cases}

In light of the preceding result, which deals with the range $2<p<\infty$, one may be curious about the endpoint cases $p=2$ and $p=\infty$. The operator $S$ (or $S_\th$) constructed in the proof admits no nice extension to the endpoints, but it is conceivable that some other map might do the job. However, we could scarcely expect to find a single superquadratic operator that obeys the required norm estimates for the whole extended range of $p$'s---that would probably be too much to hope for. Instead, we consider the two endpoints separately, asking in each case if there exists a superquadratic map $S$ (or $S_\th$) from $K^p_\th$ to $K^{p/2}_\th$ that satisfies the appropriate endpoint version of \eqref{eqn:normsf} and/or \eqref{eqn:normsftheta}. The exponents in question are thus $p=2$ and $p=\infty$; our superquadratic operators are {\it a priori} allowed to depend on $\th$ (even when $\th\in\mathcal I_0$), but the constants replacing the $B_p$'s should be absolute. 
\par The case of $p=\infty$ is actually trivial, since the map $S:H^\infty\to H^\infty$ defined by 
$$Sf=\|f\|_\infty f$$
is superquadratic, leaves $K^\infty_\th$ invariant (for each $\th\in\mathcal I$), and satisfies $\|Sf\|_\infty=\|f\|^2_\infty$. In particular, \eqref{eqn:normsf} holds with $p=\infty$ if we put $B_\infty=1$. By contrast, things become really bad at the other extreme. 

\begin{thm}\label{thm:endpeqtwo} Suppose that to each $\th\in\mathcal I_0$ there corresponds a superquadratic map $\mathcal S_\th:K^2_\th\to K^1_\th$. Then 
\begin{equation}\label{eqn:supsupinfty}
\sup_{\th\in\mathcal I_0}\sup
\left\{\|\mathcal S_\th f\|_1/\|f\|_2^2:\,f\in K^2_\th\setminus\{0\}\right\}=\infty.
\end{equation}
\end{thm}

\begin{proof} If \eqref{eqn:supsupinfty} were false, there would be an absolute constant $C>0$ such that 
\begin{equation}\label{eqn:ssfinite}
\|\mathcal S_\th f\|_1\le C\|f\|_2^2
\end{equation}
whenever $\th\in\mathcal I_0$ and $f\in K^2_\th$. Now let $a\in\D$ be a point with $|a|\ge\f12$, and put 
\begin{equation}\label{eqn:twofactors}
\th_a(z):=z\f{z-a}{1-\ov az},
\end{equation}
so that $\th_a\in\mathcal I_0$. Note also that the function 
$$f_a(z):=\f1{1-\ov az}$$
is in $K^2_{\th_a}$. In fact, this last subspace coincides with $K^1_{\th_a}$ and is two-dimensional; it is spanned by $f_a$ and the constant function $1$. Thus, writing $h_a:=\mathcal S_{\th_a}f_a$, we see that 
$$h_a(z)=\la_a+\f{\mu_a}{1-\ov az}$$
with certain coefficients $\la_a,\mu_a\in\C$. 
\par An application of \eqref{eqn:ssfinite} with $\th=\th_a$ and $f=f_a$ now yields 
\begin{equation}\label{eqn:ssfa}
\|h_a\|_1\le C(1-|a|^2)^{-1},
\end{equation}
and we are going to derive further information by estimating the left-hand side, $\|h_a\|_1$, from below. To this end, we invoke Hardy's inequality 
\begin{equation}\label{eqn:hardyineq}
\|h\|_1\ge\f1\pi\sum_{n=0}^{\infty}\f{|\widehat h(n)|}{n+1},
\end{equation}
valid for any $h\in H^1$ (see \cite[p.\,89]{G}); here $\widehat h(n)$ is the $n$th Taylor coefficient of $h$. When $h=h_a$, \eqref{eqn:hardyineq} tells us that 
\begin{equation*}
\begin{aligned}
\|h_a\|_1&\ge\f1\pi\left(|\la_a+\mu_a|
+|\mu_a|\sum_{n=1}^{\infty}\f{|a|^n}{n+1}\right)\\
&\ge\f1\pi|\la_a+\mu_a|+\f1{2\pi}|\mu_a|\log\f1{1-|a|}.
\end{aligned}
\end{equation*}
Combining this with \eqref{eqn:ssfa}, we find that 
\begin{equation}\label{eqn:lamuab}
|\la_a+\mu_a|\le\f M{1-|a|}
\end{equation}
and 
\begin{equation}\label{eqn:mulogab}
|\mu_a|\log\f1{1-|a|}\le\f M{1-|a|}
\end{equation}
with an absolute constant $M>0$; in fact, $M=2\pi C$ would do. 
\par On the other hand, because $\mathcal S_{\th_a}$ is superquadratic, we have $|h_a|\ge|f_a|^2$ on $\T$, and so 
\begin{equation}\label{eqn:hnormtwofour}
\|h_a\|_2^2\ge\|f_a\|_4^4.
\end{equation}
Parseval's identity yields 
\begin{equation*}
\begin{aligned}
\|h_a\|_2^2&=|\la_a+\mu_a|^2+|\mu_a|^2(|a|^2+|a|^4+\dots)\\
&=|\la_a+\mu_a|^2+\f{|\mu_a|^2|a|^2}{1-|a|^2}
\le|\la_a+\mu_a|^2+\f{|\mu_a|^2}{1-|a|},
\end{aligned}
\end{equation*}
while a simple computation reveals that 
$$\|f_a\|_4^4\ge\f c{(1-|a|)^3}$$
with an absolute constant $c>0$. Taking these estimates into account, we go back to \eqref{eqn:hnormtwofour} to deduce that 
\begin{equation}\label{eqn:tumba}
|\la_a+\mu_a|^2+\f{|\mu_a|^2}{1-|a|}\ge\f c{(1-|a|)^3}.
\end{equation}
At the same time, \eqref{eqn:mulogab} implies that 
$$|\mu_a|^2\le\f{M^2}{(1-|a|)^2}\left(\log\f1{1-|a|}\right)^{-2}
\le\f c{2(1-|a|)^2},$$
whenever $|a|$ is close enough to $1$. Together with \eqref{eqn:tumba}, this means that for such $a$'s we have 
$$|\la_a+\mu_a|^2\ge\f c{2(1-|a|)^3},$$
or equivalently, 
$$|\la_a+\mu_a|\ge\f{c_0}{(1-|a|)^{3/2}}$$
with $c_0:=\sqrt{c/2}$. However, for small values of $1-|a|$, this last estimate is obviously incompatible with \eqref{eqn:lamuab}. The contradiction completes the proof.
\end{proof}

\par A glance at the proof reveals that the class $\mathcal I_0$ in the theorem's statement can be actually replaced by a tiny subset thereof, namely, by the family of two-factor Blaschke products of the form \eqref{eqn:twofactors}. We now supplement Theorem \ref{thm:endpeqtwo} (and its refined version just mentioned) with another result in the same vein, which is essentially a consequence of Aleksandrov's work in \cite{A2}. This time we produce a {\it single} inner function $\th$ for which the estimate 
\begin{equation}\label{eqn:onetwoest}
\|Sf\|_1\le C\|f\|_2^2,\qquad f\in K^2_\th,
\end{equation}
fails whenever $S:K^2_\th\to K^1_\th$ is a superquadratic map and $C$ a positive constant. 

\begin{thm}\label{thm:endpeqtwosingle} There exists an inner function $\th$ such that every superquadratic operator $S:K^2_\th\to K^1_\th$ satisfies 
\begin{equation}\label{eqn:supeqinfty}
\sup\left\{\|Sf\|_1/\|f\|_2^2:\,
f\in K^2_\th\setminus\{0\}\right\}=\infty.
\end{equation}	
\end{thm}

\begin{proof} Results of \cite[Section 4]{A2} imply that there exists an inner function $\th$ and a positive Borel measure $\mu$ on $\D$ with the following properties: $K^1_\th$ embeds in $L^1(\mu)$ (meaning that 
$$\int_{\D}|g|\,d\mu\le B\|g\|_1,\qquad g\in K^1_\th,$$
with some constant $B>0$ independent of $g$), but $K^2_\th$ does not embed in $L^2(\mu)$. Now, if for that $\th$ we could find a superquadratic operator $S:K^2_\th\to K^1_\th$ satisfying \eqref{eqn:onetwoest} with some fixed $C>0$, then it would follow that
$$\int_{\D}|f|^2\,d\mu\le\int_{\D}|Sf|\,d\mu\le B\|Sf\|_1\le BC\|f\|_2^2$$
for each $f\in K^2_\th$, leading to a contradiction.
\end{proof}

\section{Extrapolation theorem: statement and proof}

In this section, we apply our main result (namely, Theorem \ref{thm:quasisq} above) to deduce an extrapolation theorem for a class of sublinear operators acting on $K^p_\th$ spaces. 

\par Suppose that $\mathcal E_1$ and $\mathcal E_2$ are two function spaces. More precisely, it will be assumed for $j=1,2$ that $\mathcal E_j$ is a vector space consisting of complex-valued functions that live on a certain set $X_j$. Recall that an operator $T:\mathcal E_1\to\mathcal E_2$ is said to be {\it sublinear} if it satisfies 
\begin{equation}\label{eqn:sublinone}
|T(f+g)|\le|Tf|+|Tg|
\end{equation}
and 
\begin{equation}\label{eqn:sublintwo}
|T(\la f)|=|\la||Tf|
\end{equation}
whenever $f,g\in\mathcal E_1$ and $\la\in\C$. 

\par Furthermore, we say that an operator $T:\mathcal E_1\to\mathcal E_2$ is {\it solid} if it has the following properties: First, 
there exists a constant $\ga>0$ such that 
\begin{equation}\label{eqn:niceone}
|Tf|^2\le\ga|T(f^2)|
\end{equation}
for every $f\in\mathcal E_1$ satisfying $f^2\in\mathcal E_1$, and secondly, 
\begin{equation}\label{eqn:nicetwo}
|TF|\le|TG|
\end{equation}
whenever $F,G\in\mathcal E_1$ are functions with $|F|\le|G|$ on $X_1$. 

\par It is understood that conditions \eqref{eqn:sublinone}--\eqref{eqn:nicetwo} above hold pointwise on $X_2$, either everywhere or almost everywhere (in the appropriate sense), depending on the context. 

\par The statement of our extrapolation theorem below involves a general measure space $(X,\mathfrak A,\mu)$, where the three symbols have the usual meaning. We write $L^p(\mu)$ for $L^p(X,\mathfrak A,\mu)$; in particular, $L^0(\mu)$ stands for the space of $\mathfrak A$-measurable functions on $X$. The notation $L^p$ (without specifying the measure) is, of course, retained for the case of $m$, the normalized Lebesgue measure on $\T$. 

\begin{thm}\label{thm:extrapol} Let $1<\si<\infty$ and $1\le\tau\le\infty$. Given an inner function $\th$ and a measure space $(X,\mathfrak A,\mu)$, suppose that $T:K^1_{\th^2}\to L^0(\mu)$ is a solid sublinear operator. Assume also that $T$ maps $K^{\si}_\th$ boundedly into $L^{\tau}(\mu)$. Then $T$ maps $K^p_\th$ boundedly into $L^q(\mu)$ whenever $\si<p<\infty$ and $p/q=\si/\tau$.
\end{thm}

\begin{proof} The case where $\tau=\infty$ is trivial, since the only possible value of $q$ is then $\infty$, and $K^p_\th\subset K^{\si}_\th$ for $p>\si$. 
\par To deal with the case $1\le\tau<\infty$, we begin by showing that $T$ acts boundedly from $K^{2\si}_\th$ to $L^{2\tau}(\mu)$. Let $S_\th$ be the superquadratic map from Theorem \ref{thm:quasisq}. Given $f\in K^{2\si}_\th$, we have then $S_\th f\in K^{\si}_\th$ and 
\begin{equation}\label{eqn:sigmakish}
\|S_\th f\|_\si\le C\|f\|_{2\si}^2,
\end{equation}
where 
$$C=C(\si,\th):=B_{2\si}\left(\f{1+|\th(0)|}{1-|\th(0)|}\right)^2;$$
also, $|S_\th f|\ge|f|^2$ on $\D$ (and $m$-almost everywhere on $\T$). For $f$ as above, we may invoke \eqref{eqn:niceone} and then \eqref{eqn:nicetwo}, with $F=f^2$ and $G=S_\th f$, to find that
\begin{equation}\label{eqn:usefulineq}
|Tf|^2\le\ga|T(f^2)|\le\ga|T(S_\th f)|
\end{equation}
$\mu$-almost everywhere on $X$. (To justify the former step, note that $f^2$ is in $K^\si_{\th^2}$ and hence in $K^1_{\th^2}$.) 
\par Raising the resulting inequality from \eqref{eqn:usefulineq} to the power $\tau$ and integrating, we get 
\begin{equation}\label{eqn:podonok}
\int_X|Tf|^{2\tau}d\mu\le\ga^{\tau}\int_X|T(S_\th f)|^{\tau}d\mu.
\end{equation}
On the other hand, we know by assumption that 
$$\int_X|Tg|^{\tau}d\mu\le M^{\tau}\|g\|_\si^\tau,\qquad g\in K^{\si}_\th,$$
with some fixed $M>0$. Applying this to $g=S_\th f$ gives
\begin{equation}\label{eqn:ublyudok}
\int_X|T(S_\th f)|^{\tau}d\mu\le M^{\tau}\|S_\th f\|_\si^\tau,
\end{equation}
and we now combine \eqref{eqn:podonok} with \eqref{eqn:ublyudok} to obtain 
$$\int_X|Tf|^{2\tau}d\mu\le(M\ga)^{\tau}\|S_\th f\|_\si^\tau.$$
In conjunction with \eqref{eqn:sigmakish}, this last estimate yields
$$\int_X|Tf|^{2\tau}d\mu\le(CM\ga)^{\tau}\|f\|_{2\si}^{2\tau},$$
proving our claim that the operator 
\begin{equation}\label{eqn:poganets}
T:\,K^p_\th\to L^q(\mu) 
\end{equation}
is bounded when $p=2\si$ and $q=2\tau$. 

\par Iterating the above argument, we arrive at a similar boundedness result for the operator \eqref{eqn:poganets} whenever $p=2^n\si(=:p_n)$ and $q=2^n\tau(=:q_n)$ for some integer $n\ge0$. The remaining cases can now be proved by interpolation. Indeed, for $1<p<\infty$, the operator $P_\th$ defined by
$$P_\th h:=P_+h-\th P_+(\ov\th h),\qquad h\in L^p,$$
where $P_+:L^p\to H^p$ is the Riesz projection (see \cite[Chapter III]{G}), is bounded on $L^p$. Moreover, $P_\th$ is a bounded projection from $L^p$ onto $K^p_\th$. Consequently, the (already established) boundedness property of the map \eqref{eqn:poganets} with $p=p_n$ and $q=q_n$ can be rephrased by saying that the sublinear operator $TP_\th:\,L^p\to L^q(\mu)$ is bounded for any such pair of exponents. The Riesz--Thorin convexity theorem, or rather its extension to sublinear operators due to Calder\'on and Zygmund (see \cite{CZ} or \cite{SW}), now guarantees that $TP_\th$ maps $L^p$ boundedly into $L^q(\mu)$ whenever the point $(1/p,1/q)$ in $\R^2$ belongs to the line segment that joins $(1/p_n,1/q_n)$ to $(1/p_{n+1},1/q_{n+1})$, for some (any) nonnegative integer $n$. In other words, $TP_\th$ is bounded as an operator from $L^p$ to $L^q(\mu)$ provided that the exponents involved satisfy $\si<p<\infty$ and $p/q=\si/\tau$. This, in turn, is equivalent to the desired conclusion. 
\end{proof}

\section{Extrapolation theorem: discussion}

In connection with our last theorem, a few comments and examples seem to be appropriate. 

\smallskip (1) First, we observe that Theorem \ref{thm:extrapol} would break down if the word \lq\lq solid" were omitted from its formulation. An example can be furnished as follows. Assume that $\th$ has infinitely many zeros, say $a_n$ ($n=1,2,\dots$), and take the (linear) differentiation operator $f\mapsto f'$ as $T$; finally, define the measure $\mu$ by 
$$d\mu(z)=(1-|z|)\,dA(z),\qquad z\in\D,$$
where $A$ is area measure on $\D$. With this choice of the main players, the operator \eqref{eqn:poganets} becomes bounded for $p=q=2$, but no such thing is true for $p=q=3$. Indeed, on the one hand, the classical Littlewood--Paley inequality 
$$\int_\D |f'(z)|^2(1-|z|)\,dA(z)\le C\|f\|_2^2$$
holds, with an absolute constant $C>0$, for all $f\in H^2$ (see, e.g., \cite{DynEnc} or \cite{G}) and hence for all $f\in K^2_\th$. On the other hand, let $f_n(z):=(1-\ov a_nz)^{-1}$ and note that $f_n\in K^3_\th$; a straightforward computation then shows that the quantity 
$$\|f_n\|_3^{-3}\int_\D |f'_n(z)|^3(1-|z|)\,dA(z)$$
behaves like a constant times $(1-|a_n|)^{-1}$ and therefore blows up as $n\to\infty$. 

\smallskip (2) To see an example where Theorem \ref{thm:extrapol} does apply, suppose that $T$ is the identity (or inclusion) map, and $\mu$ a suitable measure on the closed disk. Precisely speaking, let $\mu$ be a finite Borel measure on $\D\cup\T$ such that the singular component of $\mu\big|_\T$ assigns no mass to the set of boundary singularities for $\th$. The values of $K^p_\th$ functions are then well defined $\mu$-almost everywhere, and the boundedness issue for the operator \eqref{eqn:poganets} amounts to asking whether $K^p_\th$ embeds (continuously) in $L^q(\mu)$. The problem of characterizing such measures $\mu$ for a given $\th$ was posed, initially for $p=q=2$, by Cohn \cite{Co} and has attracted quite a bit of attention. Among the many papers that treat it, chiefly in the \lq\lq diagonal" case where $p=q$, we mention \cite{A1, A2, DAJM, DCAG, DJFA98, TV} and \cite[pp.\,80--81]{L}. See also \cite{DSpb93, DCora}, where the off-diagonal case $p<q$ was discussed (for some special measures on $\T$) in connection with the multiplicative structure of holomorphic Lipschitz spaces. 
\par The identity operator is obviously solid (in particular, \eqref{eqn:niceone} holds with $\ga=1$), so Theorem \ref{thm:extrapol} is indeed applicable in this situation. Applying it with $\tau=\si$, we recover a result of Aleksandrov (namely, \cite[Theorem 1.5]{A1}): If $1<\si<p<\infty$ and if $K^\si_\th$ embeds in $L^\si(\mu)$, then $K^p_\th$ embeds in $L^p(\mu)$. Furthermore, we know from \cite{A1, A2} that the restriction $\si>1$ in the preceding statement cannot be replaced with $\si\ge1$. Consequently, our Theorem \ref{thm:extrapol} would also become false if the endpoint $\si=1$ were included.

\smallskip (3) Finally, we remark that Theorem \ref{thm:extrapol} applies to certain maximal operators $T$. To give an example, let us associate to each point $\ze\in\T$ a set $\Om_\ze\subset\D$ and define 
$$(Tf)(\ze):=\sup\{|f(z)|:\,z\in\Om_\ze\},\qquad f\in K^1_{\th^2},$$
so that $T$ is sublinear (but not linear) and solid. Setting $\mu=m$ or perhaps considering more general measures on $\T$, we may be curious about the boundedness properties of the operator \eqref{eqn:poganets} for various values of $p$ and $q$, as soon as the sets $\Om_\ze=\Om_\ze(\th)$ are chosen appropriately. A good choice would be one where the $\Om_\ze$'s are reasonably nice, $\Om_\ze$ touches the circle at $\ze$, and the order of contact is controlled in terms of the distance from $\ze$ to the boundary singularities of $\th$.

\end{document}